\documentclass[11pt]{amsart}
\usepackage[a4paper,top=27mm,bottom=23mm,left=28mm,right=28mm]{geometry}
\usepackage{amsmath,amssymb,amsthm,hyperref}
\hypersetup{colorlinks=true,urlcolor=blue,citecolor=blue,linkcolor=blue,pdftitle={Quantitative Krein-Milman Approximation of Covering Markov Operators},pdfauthor={Hulya Duru, Serkan Ilter, Aygul Bilgin}}
\allowdisplaybreaks
\newtheorem{theorem}{Theorem}[section]
\newtheorem{proposition}[theorem]{Proposition}
\newtheorem{lemma}[theorem]{Lemma}
\newtheorem{corollary}[theorem]{Corollary}
\theoremstyle{remark}
\newtheorem{remark}[theorem]{Remark}
\newtheorem{example}[theorem]{Example}
\newcommand{\C}{\mathrm C}
\newcommand{\M}{\mathcal M}
\newcommand{\E}{\mathcal E}
\newcommand{\Prob}{\mathcal P}
\newcommand{\TV}{\mathrm{TV}}
\newcommand{\dist}{\operatorname{dist}}
\title[Quantitative Krein--Milman approximation]{Quantitative Krein--Milman Approximation of Covering Markov Operators}
\author{H\"ulya Duru}
\author{Serkan \.{I}lter}
\author{Ayg\"ul Bilgin}
\thanks{Author affiliations: H\"ulya Duru and Serkan \.{I}lter, Department of Mathematics, Faculty of Science, Istanbul University, Vezneciler, 34134 Istanbul, T\"urkiye; Ayg\"ul Bilgin, Institute of Science, Istanbul University, Istanbul, T\"urkiye. E-mail addresses: H\"ulya Duru (corresponding author), \texttt{hduru@istanbul.edu.tr}; Serkan \.{I}lter, \texttt{ilters@istanbul.edu.tr}; Ayg\"ul Bilgin, \texttt{aygulbilgin34@gmail.com}. ORCID iDs: H\"ulya Duru, 0000-0003-0689-460X; Serkan \.{I}lter, 0000-0002-7847-5124.}
\date{}
\begin{document}
\begin{abstract}
We study operator-norm approximation of positive unital operators by finite convex combinations of extreme Markov operators. Let $q:Y\to X$ be a connected degree-$d$ covering between finite connected graphs, $d\ge2$. Its averaging operator
\[
(A_qf)(x)=\frac1d\sum_{y\in q^{-1}(x)}f(y)
\]
belongs to the operator-norm closed convex hull of the extreme Markov operators but not to their finite convex hull. Let $\E_m$ consist of convex combinations of at most $m$ extreme Markov operators, and let $\E_m^{\mathrm{eq}}$ consist of averages of exactly $m$ such operators, with repetitions allowed. For $m=kd+r$, $k\ge0$, $1\le r\le d$, we prove
\[
\dist(A_q,\E_m)=\frac{2(d-r+1)}{d[k(d-r+1)+1]},\qquad
\dist(A_q,\E_m^{\mathrm{eq}})=\frac{2r(d-r+1)}{dm}.
\]
Both infima are attained. The exact distances are independent of the covering permutations. For degree-$d$ coverings of compact Hausdorff spaces without continuous sections, the right-hand sides give lower bounds for the respective distances. These bounds follow from estimates for the representing measures. For connected graph coverings, continuous maps are constructed to attain both bounds. Equal coefficients give the unrestricted minimum exactly when $k=0$ or $r\in\{1,d\}$. We compute their ratio. We also determine the asymptotic behavior of both distances; in particular, the normalized equal-coefficient distances do not converge for $d\ge3$.
\end{abstract}
\maketitle
\enlargethispage{4pt}

\section{Introduction}
We study approximation in the operator norm by finite convex combinations of extreme positive unital operators. Our setting is the convex set
\[
\M(Y,X)=\{S\in\mathcal L(\C(Y),\C(X)):S\ge0,\ S\mathbf1_Y=\mathbf1_X\},
\]
where $X,Y$ are compact Hausdorff spaces and $\C(X),\C(Y)$ are real Banach lattices with the supremum norm. We write $\C(X,Y)$ for the continuous maps $X\to Y$.
By a degree-$d$ covering we mean a continuous surjection $q:Y\to X$ that is locally equivalent over $X$ to the projection
$U\times\{1,\ldots,d\}\to U$. The covering is called connected if $Y$ is connected.
By Lemma~\ref{lem:extremes},
\[
\operatorname{Ext}\M(Y,X)=\{C_h:h\in\C(X,Y)\},\qquad C_hf=f\circ h.
\]
In every finite convex combination considered here, the coefficients are independent of $x\in X$.

Extreme positive operators were studied in \cite{Phelps63,BLP65}. Krein--Milman-type theorems for operators appear in \cite{Phelps70,MP70}; see \cite[Corollary~2.6]{MP70} for a weak-operator convex-hull result. Relevant approximation results for Markov and related operators include \cite{Brown66,HNR09,Iwanik80,Kim68,Kim72a,Kim72b}. In \cite{Kim68} and \cite[Theorem~3]{Kim72a}, the approximation norm is $L^\infty\!\to L^1$; \cite[Theorem~4]{Kim72a} uses the $L^2$ operator norm for Hilbert--Schmidt-type Markov operators. Positive contractions and positive operators on $L^p$ appear in \cite{Kim72c,Kim75,Grz90}. For normalized positive operators with weak order units, see \cite{DuruIlter18}; for related results on positive operators and vector lattices, see \cite{DuruIlter20,Kusraev24}. For convexity and Markov-operator approximation, see \cite{Olteanu22}; for a different operator-theoretic application of graph coverings, see \cite{CK23}.
In a distinct quantitative convexity problem, Kadison and Pedersen \cite{KP85}
studied finite means and convex combinations of unitaries in $C^*$-algebras.
Their results concern unitary decompositions; here we calculate the best operator-norm
errors for a fixed number of composition operators.

A finite graph is obtained by identifying endpoints of finitely many intervals $[0,1]$; the endpoints of one interval may also be identified. Let $q:Y\to X$ be a connected degree-$d$ covering between finite connected graphs, $d\ge2$. Define $A_q\in\M(Y,X)$ by
\begin{equation}\label{eq:coveraverage}
(A_qf)(x)=\frac1d\sum_{y\in q^{-1}(x)}f(y),\qquad f\in\C(Y).
\end{equation}
We call $A_q$ the covering Markov operator associated with $q$.
The covering has no continuous section. Moreover,
\[
A_q=\sum_{j=1}^N\alpha_jC_{h_j},\quad \alpha_j\ge0,\quad\sum_j\alpha_j=1
\quad\Longrightarrow\quad q\circ h_j=\operatorname{id}_X\ \ (\alpha_j>0),
\]
by equality of the representing probability measures. Hence $A_q\notin\operatorname{co}\operatorname{Ext}\M(Y,X)$.

For $m\ge1$, put
\begin{equation}\label{eq:Em}
\E_m(Y,X)=\left\{\sum_{j=1}^m\alpha_j C_{h_j}:\ \alpha_j\ge0,\ \sum_{j=1}^m\alpha_j=1,\ h_j\in\C(X,Y)\right\},
\end{equation}
\begin{equation}\label{eq:Emeq}
\E_m^{\mathrm{eq}}(Y,X)=\left\{\frac1m\sum_{j=1}^m C_{h_j}:\ h_j\in\C(X,Y)\right\}.
\end{equation}
Zero coefficients and repeated maps are allowed. All distances are in the operator norm.

\begin{theorem}\label{thm:main}
Let $q:Y\to X$ be a connected degree-$d$ covering between finite connected graphs, $d\ge2$. If $m=kd+r$, $k\ge0$, $1\le r\le d$, and $s=d-r+1$, then
\begin{equation}\label{eq:main}
\dist(A_q,\E_m(Y,X))=\frac{2s}{d(ks+1)}.
\end{equation}
The infimum is attained in $\E_m(Y,X)$.
\end{theorem}

\begin{theorem}\label{thm:eq}
Let $q:Y\to X$ be a connected degree-$d$ covering between finite connected graphs, $d\ge2$. If $m=kd+r$, $k\ge0$, $1\le r\le d$, and $s=d-r+1$, then
\begin{equation}\label{eq:eqdist}
\dist(A_q,\E_m^{\mathrm{eq}}(Y,X))=\frac{2rs}{dm}.
\end{equation}
The infimum is attained in $\E_m^{\mathrm{eq}}(Y,X)$.
\end{theorem}

The lower bounds require no graph structure. If $X,Y$ are compact Hausdorff spaces and $q:Y\to X$ is a degree-$d$ covering without a continuous section, then, for $m=kd+r$ and $s=d-r+1$,
\[
\dist(A_q,\E_m(Y,X))\ge\frac{2s}{d(ks+1)},\qquad
\dist(A_q,\E_m^{\mathrm{eq}}(Y,X))\ge\frac{2rs}{dm}.
\]
These are Proposition~\ref{prop:lower} and Lemma~\ref{lem:eqlower}. Lemmas~\ref{lem:finite-chain} and \ref{lem:eq-finite-chain} concern finite index sets. The graph assumption enters only in Lemma~\ref{lem:realize}, which yields the corresponding upper bounds.

For fixed $d$, the distances in \eqref{eq:main} and \eqref{eq:eqdist} are independent of the covering. Equal coefficients are optimal exactly when $k=0$ or $r\in\{1,d\}$. The finite-term and asymptotic comparisons are given in Section~6.

\section{Positive kernels and an exact norm formula}
Write $\Prob(Y)$ for the regular Borel probability measures on the compact Hausdorff space $Y$. For every $S\in\M(Y,X)$ there is a weak-star continuous family $(\mu_x)_{x\in X}\subset\Prob(Y)$ such that
\[
(Sf)(x)=\int_Y f\,d\mu_x.
\]
We use $\|\nu\|_{\TV}=|\nu|(Y)$ for finite signed measures.

The following classical characterization is included for completeness; see
\cite{Phelps63,BLP65}.
\begin{lemma}\label{lem:extremes}
For compact Hausdorff $X,Y$,
\[
\operatorname{Ext}\M(Y,X)=\{C_h:h\in\C(X,Y)\}.
\]
\end{lemma}
\begin{proof}
Dirac measures are extreme in $\Prob(Y)$, so $C_h$ is extreme. Conversely, let $S\in\M(Y,X)$ and suppose $\mu_{x_0}$ is not a Dirac measure. Choose $g\in\C(Y)$ and $\chi\in\C(X)$ with
\[
0\le g,\chi\le1,\quad \chi(x_0)=1,\quad
\mu_{x_0}(g^2)-\mu_{x_0}(g)^2>0.
\]
Define
\[
(Df)(x)=\chi(x)\bigl(\mu_x(gf)-\mu_x(g)\mu_x(f)\bigr).
\]
Then $D:\C(Y)\to\C(X)$ is linear, $D\mathbf1=0$, and
\[
(S\pm D)f(x)=\int_Y f(y)\bigl[1\pm\chi(x)(g(y)-\mu_x(g))\bigr]d\mu_x(y).
\]
The densities are nonnegative; hence $S\pm D\in\M(Y,X)$. Since
\[
(Dg)(x_0)=\mu_{x_0}(g^2)-\mu_{x_0}(g)^2>0,
\]
$S$ is not extreme. Finally, if every $\mu_x$ is Dirac, weak-star continuity of $x\mapsto\mu_x=\delta_{h(x)}$ implies $h\in\C(X,Y)$.
\end{proof}

For the operator in \eqref{eq:coveraverage},
\[
\mu_x=\frac1d\sum_{y\in q^{-1}(x)}\delta_y.
\]
If $R=\sum_{j=1}^m\alpha_jC_{h_j}$, then
\begin{equation}\label{eq:TV}
\|A_q-R\|=\sup_{x\in X}\left\|\frac1d\sum_{y\in q^{-1}(x)}\delta_y-\sum_{j=1}^m\alpha_j\delta_{h_j(x)}\right\|_{\TV}.
\end{equation}
If $q^{-1}(x)=\{y_1,\ldots,y_d\}$ and $\nu\in\Prob(Y)$, then
\begin{equation}\label{eq:fiber-tv}
\frac12\|\mu_x-\nu\|_{\TV}
=\sum_{i=1}^d\left(\frac1d-\nu(\{y_i\})\right)_+.
\end{equation}
This identity also applies when $\nu(Y\setminus q^{-1}(x))>0$.

\section{A universal lower bound}
The next results apply to degree-$d$ coverings of compact Hausdorff spaces without continuous sections.

\begin{lemma}\label{lem:balancing}
Let $d,m\ge1$, $m-1=kd+r-1$, $k\ge0$, $1\le r\le d$, and $c_1,\ldots,c_d\in\mathbb Z_{\ge0}$ with $\sum_i c_i\le m-1$. For every $M\ge0$,
\begin{equation}\label{eq:occupancy}
\sum_{i=1}^d\left(\frac1d-c_iM\right)_+
\ge(d-r+1)\left(\frac1d-kM\right)_+.
\end{equation}
\end{lemma}
\begin{proof}
Put $\varphi(t)=(1/d-Mt)_+$. It is nonincreasing and convex. Thus we may take $\sum_i c_i=m-1$, and for $c_i\ge c_j+2$,
\[
\varphi(c_i)+\varphi(c_j)\ge\varphi(c_i-1)+\varphi(c_j+1).
\]
Iterating gives $r-1$ entries equal to $k+1$ and $s=d-r+1$ equal to $k$. Hence
\[
\sum_i\varphi(c_i)\ge(r-1)\varphi(k+1)+s\varphi(k)\ge s\varphi(k).
\]
\end{proof}

\begin{proposition}\label{prop:lower}
Let $q:Y\to X$ be a degree-$d$ covering between compact Hausdorff spaces without a continuous section. If $m=kd+r$, $k\ge0$, $1\le r\le d$, and $s=d-r+1$, then
\[
\dist(A_q,\E_m(Y,X))\ge\frac{2s}{d(ks+1)}.
\]
\end{proposition}
\begin{proof}
Fix $R=\sum_{j=1}^m\alpha_jC_{h_j}\in\E_m$, and choose $j_0$ with $M:=\alpha_{j_0}=\max_j\alpha_j$. As $q\circ h_{j_0}\ne\operatorname{id}_X$ (otherwise $h_{j_0}$ would
be a section),
\[
\exists x_0\in X:\quad h_{j_0}(x_0)\notin q^{-1}(x_0).
\]
Write $q^{-1}(x_0)=\{y_1,\ldots,y_d\}$, and let
\[
\nu=\sum_{j=1}^m\alpha_j\delta_{h_j(x_0)},\qquad
c_i=\#\{j\ne j_0:h_j(x_0)=y_i\}.
\]
Then $\sum_i c_i\le m-1$, $\nu(\{y_i\})\le c_iM$, and, by \eqref{eq:fiber-tv} and Lemma~\ref{lem:balancing},
\[
\frac12\|\mu_{x_0}-\nu\|_{\TV}
\ge s\left(\frac1d-kM\right)_+.
\]
Moreover, with $F=q^{-1}(x_0)$ and $\nu_i=\nu(\{y_i\})$,
\[
\frac12\|\mu_{x_0}-\nu\|_{\TV}
=\nu(Y\setminus F)+\sum_{i=1}^d(\nu_i-1/d)_+
\ge\nu(Y\setminus F)\ge M.
\]
For $a=s/[d(ks+1)]$ we have $s(1/d-ka)=a$, and therefore
\[
\frac12\|A_q-R\|\ge
\max\left\{M,\ s\left(\frac1d-kM\right)_+\right\}\ge a.
\]
The last inequality follows from $M\ge a$ or $M<a$, respectively. Taking the infimum over $R$ proves the claim.
\end{proof}

A connected covering $q:Y\to X$ of finite connected graphs with $d>1$ has no continuous section: the image of a section would be a nonempty proper subset of $Y$ that is both open and closed, contradicting connectedness.

\section{Finite assignments and the upper bound}
Let $[n]=\{1,\ldots,n\}$. Fix $w=(w_j)_{j=1}^m$, $w_j\ge0$, $\sum_{j=1}^m w_j=1$. For $p=(p_j)_{j=1}^m\in[d]^m$ and $J\subset[m]$ define
\begin{equation}\label{eq:delta}
\Delta_w(p;J)=\sum_{i=1}^d\left(\frac1d-\sum_{\substack{j\notin J\\p_j=i}}w_j\right)_+,
\qquad \Delta_w(p)=\Delta_w(p;\varnothing).
\end{equation}
Write $\Delta_w(p;j)=\Delta_w(p;\{j\})$.
For $p_j=i$ and $v\ne i$, denote the replacement $p_j:=v$
by $(j:i\to v)$. If $p,p'\in[d]^m$ differ only at coordinate $j$, then
\[
\Delta_w(p)\le\Delta_w(p;j),\qquad
\Delta_w(p')\le\Delta_w(p;j)=\Delta_w(p';j).
\]
For $\pi\in\mathfrak S_d$, set $\pi p=(\pi(p_1),\ldots,\pi(p_m))$; then $\Delta_w(\pi p;J)=\Delta_w(p;J)$. We shall construct chains
\[
p^0,p^1,\ldots,p^N\in[d]^m,
\quad p^N=\pi p^0,
\quad \#\{j:p_j^{\ell-1}\ne p_j^\ell\}=1,
\quad \Delta_w(p^{\ell-1};j_\ell)\le a,
\]
where $j_\ell$ is the changed coordinate. In particular, $\Delta_w(p^\ell)\le a$ for every $\ell$ whenever $\Delta_w(p^0)\le a$.

\begin{lemma}\label{lem:finite-chain}
Let $d\ge2$, $m=kd+r$, $k\ge0$, $1\le r\le d$, and $s=d-r+1$. There exist $w\in(0,1]^m$ with $\sum_j w_j=1$ and $p^0\in[d]^m$ such that, for every $\pi\in\mathfrak S_d$, the above chain exists with
\[
a=\frac{s}{d(ks+1)}.
\]
All its terms satisfy $\Delta_w(p^\ell)\le a$, and $w$ is independent of $\pi$.
\end{lemma}
\begin{proof}
Suppose $r<d$, and set
\[
a=\frac{s}{d(ks+1)},\qquad b=\frac1{d(ks+1)}.
\]
Take pairwise disjoint sets $B_1,\ldots,B_d,\{H\},J_b\subset[m]$ with
\[
|B_i|=k,\quad |J_b|=r-1,\quad J_e=\{H\}\cup J_b,
\qquad
w_j=\begin{cases}a,&j\in\bigcup_i B_i\cup\{H\},\\b,&j\in J_b.\end{cases}
\]
Then
\begin{equation}\label{eq:ab}
(kd+1)a+(r-1)b=1,\qquad ka+b=\frac1d,\qquad sb=a.
\end{equation}
Choose $p^0\in[d]^m$ with
\[
p_j^0=i\ (j\in B_i),\qquad
p^0|_{J_e}:J_e\hookrightarrow[d],\qquad
\Delta_w(p^0)=(s-1)b\le a.
\]
For any vector obtained by changing only $J_e$-coordinates, let
$U(p)=\{p_j:j\in J_e\}$, where $|U(p)|=r$. If
$j\in J_e$, $i=p_j$, and $v\in[d]\setminus U(p)$, put
$p'=p[j\mapsto v]$. Then
\[
|U(p')|=r,\qquad \Delta_w(p;j)=s\Bigl(\frac1d-ka\Bigr)=sb=a.\]
For all these $J_e$-changes,
\[
\#\{j\in\textstyle\bigcup_{i=1}^dB_i:p_j=t\}=k\quad(t\in[d]),
\qquad p|_{J_e}:J_e\hookrightarrow[d].
\]
In addition to $(j:i\to v)$, the changes
\[
(j:i\to v),\quad(j':\ell\to i),\quad(j:v\to\ell),
\qquad i=p_j,\quad \ell=p_{j'},\quad v\notin U(p),
\]
interchange the values of $p_j,p_{j'}$ and restore the unused position $v$.
Consequently all injections $J_e\hookrightarrow[d]$ are joined by changes
satisfying $\Delta_w(p;j)=a$.

For $A,B\in\bigcup_i B_i$ with $p_A=i\ne j=p_B$, first change only
$J_e$-coordinates so that
\[
p_H=i,\qquad j\notin U(p).
\]
Now apply
\[
(A:i\to j),\qquad(B:j\to i).
\]
On removing the moving coordinate in either step, exactly $s$
positions have coefficient $ka$; all others have coefficient at least
$ka+b=1/d$. Therefore
\[
\Delta_w(p;A)=sb=a,\qquad \Delta_w(p';B)=sb=a,
\]
where $p'$ is the vector after the first change. The numbers of
$\bigcup_i B_i$-coordinates at every position are again $k$.
Reverse the preliminary $J_e$-changes. This realizes the transposition
$A\leftrightarrow B$, fixes all other coordinates, and satisfies the bound $a$.

Fix $\pi\in\mathfrak S_d$. By these $J_e$-coordinate changes,
\[
p^0\rightsquigarrow p^1,\qquad
p_j^1=\pi(p_j^0)\ (j\in J_e),\qquad
\#\{j\in\textstyle\bigcup_i B_i:p_j^1=t\}=k\quad(t\in[d]).
\]
The baseline-coordinate transpositions give
$p^1\rightsquigarrow\pi p^0$ (no further step if $k=0$).

If $r=d$, set $w_j=a=1/m$ and choose $p^0$ with
$\#\{j:p_j^0=i\}=k+1$ for each $i\in[d]$. Then $\Delta_w(p^0)=0$.
For $p_A=i\ne j=p_B$ the changes
\[
(A:i\to j),\qquad (B:j\to i),\qquad
\Delta_w(p;j_\ell)=\frac1d-\frac{k}{m}=\frac1m=a
\quad(\ell=1,2)
\]
exchange $p_A,p_B$ and fix all other coordinates. Such transpositions
produce $\pi p^0$ for every $\pi\in\mathfrak S_d$.
Finally, $\Delta_w(p^\ell)\le\Delta_w(p^{\ell-1};j_\ell)\le a$
for each step, by \eqref{eq:delta}.
\end{proof}

\begin{lemma}\label{lem:realize}
Let $q:Y\to X$ be a connected degree-$d$ covering between finite connected graphs, let $a\ge0$, and let $w\in[0,1]^m$ with $\sum_jw_j=1$. Suppose $p^0\in[d]^m$ has the following property: for every $\pi\in\mathfrak S_d$ there is a chain from $p^0$ to $\pi p^0$ of single-coordinate changes satisfying $\Delta_w(p^{\ell-1};j_\ell)\le a$. Then there exist $h_1,\ldots,h_m\in\C(X,Y)$ such that
\[
\left\|A_q-\sum_{j=1}^m w_jC_{h_j}\right\|\le2a.
\]
\end{lemma}
\begin{proof}
For $d=1$, $q$ is a homeomorphism; choosing $h_j=q^{-1}$ gives
$\sum_jw_jC_{h_j}=A_q$. Suppose $d\ge2$.

Let $V$ be the vertex set of the interval presentation of $X$. Fix
$\phi_v:[d]\to q^{-1}(v)$ bijective for $v\in V$, and require
\[
h_j(v)=\phi_v(p_j^0)\qquad (v\in V,\ j\in[m]).
\]
For an attached interval $e:[0,1]\to X$, put $u=e(0)$, $v=e(1)$;
$u=v$ is allowed. Choose $y_i\in\C([0,1],Y)$ and
$\lambda_0,\lambda_1\in\mathfrak S_d$ with
\[
q(y_i(t))=e(t),\quad
q^{-1}(e(t))=\{y_i(t):i\in[d]\},\quad
\begin{cases}
y_i(0)=\phi_u(\lambda_0(i)),\\
y_i(1)=\phi_v(\lambda_1(i)).
\end{cases}
\]
Set
\begin{equation}\label{eq:pi-orientation}
p^{\mathrm{in}}=\lambda_0^{-1}p^0,\quad
p^{\mathrm{out}}=\lambda_1^{-1}p^0,\quad
\pi_e=\lambda_1^{-1}\circ\lambda_0,\quad
p^{\mathrm{out}}=\pi_ep^{\mathrm{in}}.
\end{equation}
Since $d\ge2$ and $m\ge1$, some $\pi\in\mathfrak S_d$ satisfies
$\pi p^0\ne p^0$. The first step of its chain implies
\[
\Delta_w(p^0)\le\Delta_w(p^0;j_1)\le a.
\]
The assumed chains $p^0\rightsquigarrow p^{\mathrm{in}}$ and
$p^0\rightsquigarrow p^{\mathrm{out}}$ give, by reversing the first,
\[
p^{\mathrm{in}}=p^0_e,\ p^1_e,\ldots,p^N_e=p^{\mathrm{out}},\qquad
\Delta_w(p_e^{\ell-1};j_\ell)\le a.
\]
Indeed, $\Delta_w(p;j)=\Delta_w(p';j)$ when $p,p'$ differ only at $j$.

Choose $0<a_1<b_1<\cdots<a_N<b_N<1$ and set
\[
I_\ell=[a_\ell,b_\ell],\quad
G_0=[0,a_1],\quad G_\ell=[b_\ell,a_{\ell+1}]\ (1\le\ell<N),
\quad G_N=[b_N,1].
\]
If $N=0$, take $G_0=[0,1]$. When
$p^{\ell-1}_{e,j_\ell}=i$ and $p^\ell_{e,j_\ell}=i'$, choose
$\gamma_\ell\in\C(I_\ell,Y)$ satisfying
\[
\gamma_\ell(a_\ell)=y_i(a_\ell),\qquad
\gamma_\ell(b_\ell)=y_{i'}(b_\ell).
\]
Such paths exist since $Y$ is a connected finite graph. For every $j\in[m]$,
define
\[
h_j(e(t))=
\begin{cases}
y_{p^\ell_{e,j}}(t),&t\in G_\ell,\quad0\le\ell\le N,\\
\gamma_\ell(t),&t\in I_\ell,\quad j=j_\ell,\\
y_{p^{\ell-1}_{e,j}}(t),&t\in I_\ell,\quad j\ne j_\ell.
\end{cases}
\]
At the endpoints,
\[
h_j(e(0))=\phi_u(p_j^0),\qquad h_j(e(1))=\phi_v(p_j^0).
\]
Hence these prescriptions define $h_j\in\C(X,Y)$, even for $u=v$.

Let $\mu_t=d^{-1}\sum_{i=1}^d\delta_{y_i(t)}$ and
$\nu_t=\sum_{j=1}^m w_j\delta_{h_j(e(t))}$. By \eqref{eq:fiber-tv},
\[
\frac12\|\mu_t-\nu_t\|_{\TV}
\le\begin{cases}
\Delta_w(p_e^\ell),&t\in G_\ell,\\
\Delta_w(p_e^{\ell-1};j_\ell),&t\in I_\ell,
\end{cases}
\le a.
\]
The first bound follows from the exact point coefficients; the second
uses only the coordinates other than $j_\ell$.
Taking the supremum over all attached intervals and applying
\eqref{eq:TV} yields $\|A_q-\sum_jw_jC_{h_j}\|\le2a$.
\end{proof}

\begin{proof}[Proof of Theorem~\ref{thm:main}]
Set $a=s/[d(ks+1)]$. Proposition~\ref{prop:lower} gives $\dist(A_q,\E_m)\ge2a$, while Lemmas~\ref{lem:finite-chain} and \ref{lem:realize} give $R\in\E_m$ with $\|A_q-R\|\le2a$. Thus
\[
\dist(A_q,\E_m)=\|A_q-R\|=2a.
\]
\end{proof}

\begin{corollary}\label{cor:maxweight}
Let $q:Y\to X$ satisfy the assumptions of Theorem~\ref{thm:main}, and let $m=kd+r$, $k\ge1$, $1\le r\le d$. If $R=\sum_{j=1}^m\alpha_jC_{h_j}\in\E_m(Y,X)$ satisfies $\|A_q-R\|=\dist(A_q,\E_m)$, then
\[
\max_{1\le j\le m}\alpha_j=\frac{d-r+1}{d[k(d-r+1)+1]}.
\]
\end{corollary}
\begin{proof}
Put $M=\max_j\alpha_j$, $s=d-r+1$ and $a=s/[d(ks+1)]$. Proposition~\ref{prop:lower} yields
\[
a=\frac12\|A_q-R\|\ge\max\left\{M,\ s\left(\frac1d-kM\right)_+\right\}.
\]
Thus $M\le a$. If $M<a$ and $k\ge1$, then $s(1/d-kM)>s(1/d-ka)=a$, a contradiction. Hence $M=a$.
For $k=0$, the second term in the preceding lower estimate does not depend on $M$; this is why $k\ge1$ is assumed.
\end{proof}

\section{Equal-weight approximation}
We use the same finite-coordinate notation as in Section~4.

\begin{lemma}\label{lem:eqlower}
Let $q:Y\to X$ be a degree-$d$ covering between compact Hausdorff spaces without a continuous section. If $m=kd+r$, $k\ge0$, $1\le r\le d$, and $s=d-r+1$, then
\[
\dist(A_q,\E_m^{\mathrm{eq}}(Y,X))\ge\frac{2rs}{dm}.
\]
\end{lemma}
\begin{proof}
Fix $R=m^{-1}\sum_{j=1}^mC_{h_j}\in\E_m^{\mathrm{eq}}$. There is $x_0\in X$ with $h_1(x_0)\notin q^{-1}(x_0)$. Write $q^{-1}(x_0)=\{y_1,\ldots,y_d\}$, and put
\[
c_i=\#\{j\in\{2,\ldots,m\}:h_j(x_0)=y_i\}.
\]
Then $\sum_i c_i\le m-1$, and Lemma~\ref{lem:balancing}, with $M=1/m$, gives
\[
\begin{aligned}
\frac12\|A_q-R\|
&\ge\sum_{i=1}^d\left(\frac1d-\frac{c_i}{m}\right)_+
\ge s\left(\frac1d-\frac{k}{m}\right)
=\frac{rs}{dm}.
\end{aligned}
\]
\end{proof}

\begin{lemma}\label{lem:eq-finite-chain}
Let $d\ge2$, $m=kd+r$, $k\ge0$, $1\le r\le d$, and $w_j=1/m$ for $j\in[m]$. There is $p^0\in[d]^m$ such that for every $\pi\in\mathfrak S_d$ a chain of single-coordinate changes joins $p^0$ to $\pi p^0$ and satisfies
\[
\Delta_w(p^{\ell-1};j_\ell)\le a_{\mathrm{eq}}:=\frac{r(d-r+1)}{dm}
\quad(1\le\ell\le N).
\]
\end{lemma}
\begin{proof}
Suppose $r<d$. For $p\in[d]^m$, let
\[
n_i(p)=\#\{j:p_j=i\},\qquad
F(p)=\{i:n_i(p)=k+1\},\qquad L(p)=[d]\setminus F(p).
\]
Choose $p^0$ with $n_i(p^0)\in\{k,k+1\}$ and
$|F(p^0)|=r$, $|L(p^0)|=d-r$. Then
\[
\Delta_w(p^0)=\frac{r(d-r)}{dm}\le a_{\mathrm{eq}}.
\]
For $p$ with these counts, $i\in F(p)$, $\ell\in L(p)$, and
$p_j=i$, put $p'=p[j\mapsto\ell]$. Then
\[
\begin{aligned}
F(p')&=(F(p)\setminus\{i\})\cup\{\ell\},\qquad
n_t(p')\in\{k,k+1\}\quad(t\in[d]),\\
\Delta_w(p;j)&=(d-r+1)\Bigl(\frac1d-\frac{k}{m}\Bigr)
=\frac{r(d-r+1)}{dm}=a_{\mathrm{eq}}.
\end{aligned}
\]
Given $\pi\in\mathfrak S_d$, put $F^*=\pi(F(p^0))$.
For $F(p)\ne F^*$ choose
$i\in F(p)\setminus F^*$ and $\ell\in F^*\setminus F(p)$.
The same coordinate change gives
\[
|F(p')\mathbin\triangle F^*|=|F(p)\mathbin\triangle F^*|-2.
\]
After finitely many changes we reach $p$ with
\[
n_i(p)=n_i(\pi p^0)\qquad (i\in[d]).
\]

For $p_A=i\ne\ell=p_B$, the following changes interchange $p_A,p_B$
and leave every other coordinate unchanged (exchange $A,B$ first if
$i\in L(p)$, $\ell\in F(p)$):
\[
\begin{array}{ll}
 i\in F(p),\ \ell\in L(p):
 &(A:i\to\ell),\ (B:\ell\to i);\\[2pt]
 i,\ell\in F(p),\ v\in L(p):
 &(A:i\to v),\ (B:\ell\to i),\ (A:v\to\ell);\\[2pt]
 i,\ell\in L(p),\ h\in F(p),\ p_Z=h:
 &(Z:h\to i),\ (A:i\to\ell),\ (B:\ell\to i),\ (Z:i\to h).
\end{array}
\]
Every displayed step is a change from $F(p)$ to $L(p)$; hence its value of
$\Delta_w$ is $a_{\mathrm{eq}}$. The third row requires $k\ge1$;
for $k=0$ no coordinate lies in $L(p)$.
Since the coordinate counts now agree, there exists a permutation
$\sigma\in\mathfrak S_m$ such that
\[
(\pi p^0)_j=p_{\sigma(j)}\qquad (j\in[m]).
\]
Writing $\sigma$ as a product of transpositions and using the displayed
changes when the two coordinate values differ gives
$p\rightsquigarrow\pi p^0$.

For $r=d$, $m=(k+1)d$ and $w_j=1/m$. Choose
$n_i(p^0)=k+1$ for all $i\in[d]$. If $p_A=i\ne\ell=p_B$, then
\[
(A:i\to\ell),\qquad(B:\ell\to i),\qquad
\Delta_w(p;j_t)=\frac1d-\frac{k}{m}=\frac1m=a_{\mathrm{eq}}
\quad(t=1,2).
\]
These transpositions give $\pi p^0$ for every $\pi\in\mathfrak S_d$.
\end{proof}

For $d=3$, $m=5$, $w_j=1/5$, and $\pi=(1\ 3)$, the bound is
$a_{\mathrm{eq}}=4/15$, and an explicit chain is
\[
\begin{aligned}
(1,1,2,2,3)&\xrightarrow{1:1\to3}(3,1,2,2,3)
\xrightarrow{5:3\to1}(3,1,2,2,1)\\
&\xrightarrow{2:1\to3}(3,3,2,2,1)=\pi(1,1,2,2,3).
\end{aligned}
\]
At each step $\Delta_w(p^{\ell-1};j_\ell)=4/15$.

\begin{proof}[Proof of Theorem~\ref{thm:eq}]
Put $a_{\mathrm{eq}}=rs/(dm)$. Lemma~\ref{lem:eqlower} gives the lower bound $2a_{\mathrm{eq}}$. Lemmas~\ref{lem:eq-finite-chain} and \ref{lem:realize} give $h_j\in\C(X,Y)$ satisfying
\[
\left\|A_q-\frac1m\sum_{j=1}^mC_{h_j}\right\|\le2a_{\mathrm{eq}}.
\]
Thus the infimum equals $2a_{\mathrm{eq}}$ and is attained.
\end{proof}

\section{Consequences and examples}
\begin{corollary}\label{cor:closure}
Let $q:Y\to X$ be a connected degree-$d$ covering between finite connected graphs, $d\ge2$. Then
\[
A_q\in\overline{\operatorname{co}}^{\|\cdot\|}\operatorname{Ext}\M(Y,X),\qquad
A_q\notin\operatorname{co}\operatorname{Ext}\M(Y,X),
\]
and $A_q$ is not compact.
\end{corollary}
\begin{proof}
By \eqref{eq:main}, $0<\dist(A_q,\E_m)\to0$. The operator $q^*:\C(X)\to\C(Y)$, $(q^*g)(y)=g(q(y))$, is an isometric embedding and $A_qq^*=I_{\C(X)}$. Since $d>1$ and $Y$ is connected, $X$ contains an edge, and hence $\C(X)$ is infinite dimensional. Therefore $I_{\C(X)}$ is not compact, so $A_q$ cannot be compact.
\end{proof}

\begin{corollary}\label{cor:eqopt}
Under the assumptions of Theorem~\ref{thm:main},
\[
\frac{\dist(A_q,\E_m^{\mathrm{eq}})}{\dist(A_q,\E_m)}
=\frac{r[k(d-r+1)+1]}{kd+r}.
\]
The ratio equals $1$ exactly when $k=0$ or $r\in\{1,d\}$, and is greater than $1$ exactly when $k\ge1$ and $2\le r\le d-1$. If $r=1$ or $r=d$, every optimal representation $\sum_{j=1}^m\alpha_jC_{h_j}$ has $\alpha_j=1/m$ for all $j$.
\end{corollary}
\begin{proof}
By \eqref{eq:eqdist} and \eqref{eq:main},
\[
\frac{\dist(A_q,\E_m^{\mathrm{eq}})}{\dist(A_q,\E_m)}-1
=\frac{k(r-1)(d-r)}{kd+r}.
\]
For $r\in\{1,d\}$ the optimal distance is $2/m$. The estimate $\|A_q-R\|\ge2\max_j\alpha_j$ in the proof of Proposition~\ref{prop:lower} implies $\max_j\alpha_j\le1/m$ for an optimal $R$. Since $\sum_j\alpha_j=1$, all $\alpha_j=1/m$.
\end{proof}

\begin{corollary}\label{cor:asymptotic}
For every connected degree-$d$ covering between finite connected graphs, $d\ge2$,
\[
\lim_{m\to\infty}m\dist(A_q,\E_m)=2.
\]
For $m=kd+r$ and fixed $r\in\{1,\ldots,d\}$,
\[
m\dist(A_q,\E_m^{\mathrm{eq}})=\frac{2r(d-r+1)}d.
\]
Consequently,
\[
\liminf_{m\to\infty}m\dist(A_q,\E_m^{\mathrm{eq}})=2,\qquad
\limsup_{m\to\infty}m\dist(A_q,\E_m^{\mathrm{eq}})
=\frac2d\left\lfloor\frac{(d+1)^2}{4}\right\rfloor.
\]
The sequence $m\dist(A_q,\E_m^{\mathrm{eq}})$ converges if and only if $d=2$.
\end{corollary}
\begin{proof}
For $m=kd+r$ and $s=d-r+1$,
\[
\frac{2ms}{d(ks+1)}\longrightarrow2\quad(k\to\infty),
\qquad
\min_{1\le r\le d}r(d-r+1)=d,\quad
\max_{1\le r\le d}r(d-r+1)=\left\lfloor\frac{(d+1)^2}{4}\right\rfloor.
\]
Each remainder $r$ occurs infinitely often. The minimum and maximum coincide precisely for $d=2$.
\end{proof}

\begin{corollary}\label{cor:ratio-sup}
For every connected degree-$d$ covering between finite connected graphs, $d\ge2$,
\[
\sup_{m\ge1}\frac{\dist(A_q,\E_m^{\mathrm{eq}})}{\dist(A_q,\E_m)}
=\frac1d\left\lfloor\frac{(d+1)^2}{4}\right\rfloor
\sim\frac d4\quad(d\to\infty).
\]
\end{corollary}
\begin{proof}
For fixed $r$ and $s=d-r+1$, put $R_k=r(ks+1)/(kd+r)$. Then
\[
R_{k+1}-R_k=
\frac{r(rs-d)}{[d(k+1)+r](dk+r)}\ge0,
\quad rs-d=(r-1)(d-r)\ge0,
\quad \lim_{k\to\infty}R_k=\frac{rs}{d}.
\]
Maximizing $rs$ over $1\le r\le d$ gives the result.
\end{proof}

\begin{example}\label{ex:covering-permutations}
Let $X$ be the wedge of two circles. Take a connected degree-$3$ covering whose two circuit permutations are a transposition and a $3$-cycle. For $d=3$, $m=5=1\cdot3+2$,
\[
\dist(A_q,\E_5(Y,X))=\frac{2\cdot2}{3(1\cdot2+1)}=\frac49.
\]
The value is the same for every connected degree-$3$ covering between finite connected graphs.
\end{example}

\begin{example}\label{ex:ratio-example}
For $d=5$, $m=8=1\cdot5+3$,
\[
\dist(A_q,\E_8)=\frac3{10},\qquad
\dist(A_q,\E_8^{\mathrm{eq}})=\frac9{20},\qquad
\frac{\dist(A_q,\E_8^{\mathrm{eq}})}{\dist(A_q,\E_8)}=\frac32.
\]
\end{example}

\begin{remark}\label{rem:connected}
The assumption that $Y$ is connected cannot be omitted in general. If a degree-$d$ covering admits continuous sections $h_1,\ldots,h_d$ whose images partition each fiber, then
\[
A_q=\frac1d\sum_{j=1}^d C_{h_j},\qquad
\dist(A_q,\E_m)=0\quad(m\ge d).
\]
Proposition~\ref{prop:lower} assumes only the absence of a continuous section.
\end{remark}

For coverings of compact Hausdorff spaces without continuous sections,
Proposition~\ref{prop:lower} and Lemma~\ref{lem:eqlower} give the same
lower bounds, but Lemma~\ref{lem:realize} does not establish their
attainment outside the finite-graph setting. No equality for general
compact Hausdorff coverings is asserted here.

\section*{Statements and Declarations}
\noindent\textbf{Funding.} The authors received no financial support for this research.\par
\noindent\textbf{Competing interests.} The authors declare that they have no competing interests.\par
\noindent\textbf{Data availability.} No datasets were generated or analysed during this theoretical study.\par


\begin{thebibliography}{99}
\footnotesize\setlength{\itemsep}{-1pt}\setlength{\baselineskip}{9.2pt}
\bibitem{BLP65} Blumenthal, R.M., Lindenstrauss, J., Phelps, R.R.: Extreme operators into $C(K)$. Pacific J. Math. 15, 747--756 (1965). \url{https://doi.org/10.2140/pjm.1965.15.747}.
\bibitem{Brown66} Brown, J.R.: Approximation theorems for Markov operators. Pacific J. Math. 16, 13--23 (1966). \url{https://doi.org/10.2140/pjm.1966.16.13}.
\bibitem{CK23} Cimasoni, D., Kassel, A.: Graph coverings and twisted operators. Algebraic Combin. 6(1), 75--94 (2023). \url{https://doi.org/10.5802/alco.258}.
\bibitem{DuruIlter18} Duru, H., \.{I}lter, S.: Weak Markov operators. Filomat 32(15), 5453--5457 (2018). \url{https://doi.org/10.2298/FIL1815453D}.
\bibitem{DuruIlter20} Duru, H., \.{I}lter, S.: Pre-Markov operators. Turkish J. Math. 44(6), 2166--2173 (2020). \url{https://doi.org/10.3906/mat-1910-59}.
\bibitem{Grz90} {\fontencoding{T1}\selectfont Grz\k{a}\'slewicz}, R.: Approximation theorems for positive operators on $L^p$-spaces. J. Approx. Theory 63(2), 123--136 (1990). \url{https://doi.org/10.1016/0021-9045(90)90097-A}.
\bibitem{HNR09} Hadwin, D., Nordgren, E., Radjavi, H.: Stochastic operators and extreme points. Oper. Matrices 3(2), 227--233 (2009). \url{https://doi.org/10.7153/oam-03-13}.
\bibitem{Iwanik80} Iwanik, A.: Approximation theorems for stochastic operators. Indiana Univ. Math. J. 29, 415--425 (1980); erratum 30, 319 (1981). \url{https://doi.org/10.1512/iumj.1980.29.29030}.
\bibitem{KP85} Kadison, R.V., Pedersen, G.K.: Means and convex combinations of unitary operators. Math. Scand. 57, 249--266 (1985). \url{https://doi.org/10.7146/math.scand.a-12116}.
\bibitem{Kim68} Kim, C.-W.: Uniform approximation of doubly stochastic operators. Pacific J. Math. 26, 515--527 (1968). \url{https://doi.org/10.2140/pjm.1968.26.515}.
\bibitem{Kim72a} Kim, C.-W.: Approximation theorems for Markov operators. Z. Wahrscheinlichkeitstheorie verw. Gebiete 21, 207--214 (1972). \url{https://doi.org/10.1007/BF00538394}.
\bibitem{Kim72b} Kim, C.-W.: Approximations of doubly substochastic operators. Michigan Math. J. 19, 83--95 (1972). \url{https://doi.org/10.1307/mmj/1029000801}.
\bibitem{Kim72c} Kim, C.-W.: Approximations of positive contractions on $L^\infty[0,1]$. Z. Wahrscheinlichkeitstheorie verw. Gebiete 24, 335--337 (1972). \url{https://doi.org/10.1007/BF00679138}.
\bibitem{Kim75} Kim, C.-W.: Approximations of positive contractions on $L^\infty$. II. Math. Ann. 217, 211--220 (1975). \url{https://doi.org/10.1007/BF01436172}.
\bibitem{Kusraev24} Kusraev, A.G.: On extreme extension of positive operators. Vladikavkaz. Mat. Zh. 26(2), 47--53 (2024). \url{https://doi.org/10.46698/s3201-6067-0570-n}.
\bibitem{MP70} Morris, P.D., Phelps, R.R.: Theorems of Krein--Milman type for certain convex sets of operators. Trans. Amer. Math. Soc. 150, 183--200 (1970). \url{https://doi.org/10.1090/S0002-9947-1970-0262804-3}.
\bibitem{Olteanu22} Olteanu, O.: Convexity, Markov operators, approximation, and related optimization. Mathematics 10(15), 2775 (2022). \url{https://doi.org/10.3390/math10152775}.
\bibitem{Phelps63} Phelps, R.R.: Extreme positive operators and homomorphisms. Trans. Amer. Math. Soc. 108, 265--274 (1963). \url{https://doi.org/10.1090/S0002-9947-1963-0156224-6}.
\bibitem{Phelps70} Phelps, R.R.: Theorems of Krein--Milman type for certain convex sets of functions and operators. Ann. Inst. Fourier (Grenoble) 20(2), 45--54 (1970). \url{https://doi.org/10.5802/aif.351}.
\end{thebibliography}
\end{document}